\documentclass[12pt]{article}
\usepackage{graphicx, tikz, color,url}
\usepackage{amsmath,amssymb,latexsym,amsthm}
\usepackage{amsfonts}
\usetikzlibrary{arrows}
\usetikzlibrary{positioning, quotes}
\usepackage{soul}

\newtheorem{theorem}{Theorem}[section]

\newtheorem{proposition}[theorem]{Proposition}

\newtheorem{conjecture}[theorem]{Conjecture}
\newtheorem{corollary}[theorem]{Corollary}

\newtheorem{problem}[theorem]{Problem}

\newcommand{\io}{\iota}

\begin{document}
\title{Isolation number: Cartesian and lexicographic products and generalized Sierpi\'{n}ski graphs}
\author{
Boštjan Brešar$^{a,b}$
\and
Tanja Dravec$^{a,b}$
\and
Daniel P. Johnston$^{c}$
\and
Kirsti Kuenzel$^{c}$\\
\and
Douglas F.\ Rall$^{d}$\\
\and
Aleksandra Tepeh$^{b,e}$
}
\maketitle

\begin{center}
$^a$ Faculty of Natural Sciences and Mathematics, University of Maribor, Slovenia\\
$^b$ Institute of Mathematics, Physics and Mechanics, Ljubljana, Slovenia\\
$^c$ Department of Mathematics, Trinity College, Hartford, CT, USA\\
$^d$ Emeritus Professor of Mathematics, Furman University, Greenville, SC, USA\\
$^e$ Faculty of Electrical Engineering and Computer Science, University of Maribor, Slovenia\\
\end{center}
\medskip

\begin{abstract}
 The isolation number $\iota(G)$ of a graph $G$ is the minimum cardinality of a set $A\subset V(G)$ such that the subgraph induced by the vertices that are not in the union of the closed neighborhoods of vertices in $A$ has no edges. The invariant, known also under the name vertex-edge domination number of $G$, has attracted a lot of interest in recent years.  In this paper, we study the behavior of the isolation number under several graph operations, namely the Cartesian and the lexicographic product and the fractalization leading to generalized Sierpi\'{n}ski graphs.  We prove several upper and lower bounds on the isolation number of the Cartesian product of two graphs. We prove a lower bound for the isolation number of the prism $G\,\Box\, K_2$ over an arbitrary graph $G$, which in the case of bipartite graphs leads to the equality $\iota(G\,\Box\, K_2)=\gamma(G)$, where $\gamma(G)$ is the domination number of $G$. In particular, $\iota(Q_{n+1})=\gamma(Q_n)$ holds for all positive integers $n$, where $Q_n$ is the $n$-dimensional hypercube. For the lexicographic product $G\circ H$ we prove that its isolation number, under certain mild restrictions, equals the total domination number of the first factor $G$. We also prove sharp lower and upper bounds on the isolation numbers of the generalized Sierpi\'{n}ski graphs $S_G^t$, where $G$ is an arbitrary base graph. These bounds in the case of classical Sierpi\'{n}ski graphs, namely $S_{K_n}^t$, coincide and lead to the exact values $\iota(S_{K_n}^t)=(n-1)\cdot n^{t-2}$ for all dimensions $t\ge 2$. 
\end{abstract}


\noindent
{\bf AMS Subj.\ Class.\ (2020)}: 05C57, 05C69

 \section{Introduction}
\label{sec: intro}
In 2017, Caro and Hansberg~\cite{ch-2017} introduced the concept of isolation in graphs, which was later studied by a number of researchers. The most fundamental version of isolation in graphs is defined as follows. Let $G$ be a graph, and $A\subset V(G)$. Denote by $N[A]$ the set of vertices that are in $A$ or adjacent to a vertex in $A$. The set $A$ is an {\em isolating set} in $G$ if $V(G)-N[A]$ is an independent set in $G$ (that is, there are no edges between vertices in $V(G)-N[A]$). The minimum cardinality of an isolating set in $G$ is the {\em isolation number}, $\iota(G)$, of $G$. Note that the condition for a set to be isolating can be rephrased as $G-N[A]$ is $K_2$-free, and by extending this condition to $\cal F$-free graphs, where $\cal F$ is an arbitrary family of graphs, one gets the general definition of $\cal F$-isolation. One can view isolation as an extension of graph domination, which is one of the most studied topics of graph theory. Recall that the {\em domination number}, $\gamma(G)$, is the minimum cardinality of a {\em dominating set} in $G$, which is a set $A$ such that $N[A]=V(G)$. Clearly, $\gamma(G)\ge \iota(G)$ in any graph $G$. 

The isolation number was studied from various perspectives and under different names. Already in 2010, Lewis et al.~\cite{lewis}
 introduced the concept of vertex-edge domination, which coincides with the basic version of isolation, and they characterized the trees $T$ with $\iota(T)=\gamma(T)$. Caro and Hansberg~\cite{ch-2017} (and independently \.{Z}yli\'{n}ski~\cite{z-2019}) proved 
a general upper bound, $\iota(G)\le n(G)/3$, which holds for all connected graphs whose order $n(G)$ is at least $6$. They improved this bound for several well known classes of graphs, such as
trees, bipartite graphs, Cartesian products of cycles, and maximal outerplanar graphs (see~\cite{tjk-2019} for a refinement concerning the latter class of graphs). Lema\'{n}ska et al.~\cite{lms-2024} characterized all unicyclic graphs and all block graphs $G$ that attain the value $\iota(G)=n(G)/3$. Furthermore, Boyer and Goddard in~\cite{bg-2024} proved a structural characterization of all graphs $G$ attaining that value. Ziemann and \.{Z}yli\'{n}ski proved the bound $\iota(G)\le 10n(G)/31$ for all cubic graphs $G$ with $n(G)>6$, and conjectured that the bound can be improved to $\iota(G)\le 2n(G)/7$, which would be sharp if true~\cite{zz-2020}. 
On the other hand, several lower bounds on $\iota(G)$ were proved in~\cite{ch-2017} when $G$ is a connected graph or belongs to some specific class of graphs. 

In this paper, we continue the investigation of isolation and focus on its behavior under several well known graph operations. Cartesian product is one of the four standard graph products~\cite{HIK}, and is arguably the most studied operation related to domination in graphs. Notably, the famous Vizing's conjecture, which states that $\gamma(G\,\Box\, H)\ge \gamma(G)\gamma(H)$ holds for every two graphs $G$ and $H$, is still unresolved despite numerous efforts in the last 50+ years; see two surveys~\cite{BDG-12} and~\cite[Chapter 18]{HHH3} and the references therein. Cartesian products, each factor of which is $K_2$, are hypercubes. Determining the domination number of an arbitrary hypercube is another challenging open problem, which motivates us to study their isolation numbers. Our motivation for studying isolation number in the class of lexicographic products came while we were searching for families of graphs in which isolation and domination numbers are the same, and it turned out that almost all lexicographic products have this property. Yet another rather well known class of graphs are the (generalized) Sierpi\'{n}ski graphs, originally introduced in relation to the Tower of Hanoi game; see a recent extensive survey~\cite{kmz-16+} where many other applications are presented. The main feature of (generalized) Sierpi\'{n}ski graphs is their fractal-like nature, and can be considered as a basic discrete version of fractals. 
While domination and total domination numbers are computationally hard problems in general graphs~\cite{HHH3}, they can be efficiently determined in several classes of graphs, in particular, in the (generalized) Sierpi\'nski graphs. Exact formulas for the domination numbers of Sierpi\'{n}ski graphs were established in~\cite{kmp-2003}; moreover, the domination number of generalized Sierpi\'{n}ski graphs was   studied in~\cite{rre}. More recently also the exact total domination numbers of arbitrary Sierpi\'nski graphs were proven~\cite{gkmmp-13}. A number of other graph invariants were studied in generalized Sierpi\'{n}ski graphs~\cite{er,erv,gs,kmz-16+,pjk}, which makes them one of the most popular classes of graphs. 

The paper is organized as follows.   In the next section, we give formal definitions of the concepts considered in the paper along with some preliminary observations. Section~\ref{sec:Cartesian} is concerned with the Cartesian product operation. We prove two upper bounds on the isolation number of the Cartesian product of two graphs and show that they are incomparable. For the first upper bound, we introduce the concept of the isolation graph of a graph, which may be of independent interest. The second upper bound on $\iota(G\,\Box\, H)$ is simply the product of vertex cover numbers of both factors. As mentioned in the context of domination number and Vizing's conjecture, lower bounds on the domination number (in Cartesian products) are in general more challenging. We obtain a lower bound on $\iota(G\,\Box\, H)$ by using the $2$-packing number of one factor and the isolation number of the other factor. For the prism $G\,\Box\, K_2$ over a graph $G$, we prove that $\iota(G\,\Box\, K_2)\ge\gamma(G)$, which holds as equality in bipartite graphs. The latter result leads to the formula $\iota(Q_{n+1})=\gamma(Q_n)$, which holds for an arbitrary hypercube $Q_n$. The main result of Section~\ref{sec:lex} is that the isolation number of the lexicographic product $G\circ H$ equals the total domination number,  $\gamma_t(G)$,  of $G$ as soon as $G$ and $H$ are connected graphs and $\iota(H)\ge 2$. In Section~\ref{sec:Sierp}, we consider the isolation numbers of generalized Sierpi\'{n}ski graphs $S_G^t$, where $G$ is an arbitrary base graph and $t$ is the dimension of $S_G^t$. For an arbitrary graph $G$, we prove a lower and an upper bound on $\iota(S_G^t)$, which rely on two invariants of the $2$-dimensional Sierpi\'{n}ski graph, $S_G^2$, over $G$. The invariants coincide in some graphs $G$, which in such cases leads to the exact values. In particular, we obtain the formula for the standard Sierpi\'{n}ski graphs $S_n^t$; namely, $\io(S_{K_n}^t)=(n-1)\cdot n^{t-2}$, which holds for all $t\geq 2$ and $n\ge 3$. We conclude the paper with several open problems.

\section{Definitions and preliminary results}
\label{sec:prelim}

In this paper, we consider  finite simple graphs. Let $G$ be a graph, and $v\in V(G)$. The {\em neighborhood}, $N(v)$, of $v$ is the set of vertices $u$ in $G$ such that $uv\in E(G)$. The {\em degree} of $v$, $d(v)$, is the cardinality of $N(v)$; that is,  $d(v)=|N(v)|$.  The {\em closed neighborhood} of $v$ is defined as $N[v]=N(v)\cup\{v\}$, while $N[A]=\cup_{v\in A}{N[v]}$. If $A\subseteq V(G)$, then by $G[A]$ we denote the subgraph of $G$ induced by $A$.
As mentioned earlier, a set $D\subset V(G)$ is a dominating set in $G$ if $N[D]=V(G)$. 
The following expressions are common in domination theory and will be used throughout. A set $A\subseteq V(G)$ {\em dominates} (vertices of) $N[A]$, and if $B\subseteq N[A]$, $B$ is {\em dominated by} $A$. In particular, a vertex $v$ {\em dominates} $N[v]$, and so every vertex adjacent to $v$ is {\em dominated by} $v$. 
A {\em $2$-packing} in $G$ is a set of vertices whose closed neighborhoods are pairwise disjoint. The {\em $2$-packing number}, $\rho_2(G),$ of $G$ is the largest cardinality of a $2$-packing in $G$. A set $D\subset V(G)$ is an {\em isolating set} in $G$ if $G-N[D]$ is either an empty graph or an edgeless graph, and $\iota(G)$ is the minimum cardinality of an isolating set in $G$.

A set $A$ is {\em independent} if no two vertices in $A$ are adjacent. The maximum cardinality of an independent set in $G$ is the {\em independence number}, $\alpha(G)$, of $G$. A set $B\subset V(G)$ is a {\em vertex cover} in $G$ if every edge in $G$ is incident with a vertex in $B$, and the minimum cardinality of a vertex cover in $G$ is the {\em vertex cover number}, $\beta(G)$, of $G$. By the Gallai Theorem, $\alpha(G)+\beta(G)=n(G)$ for any graph $G$. A set $M\subset E(G)$ is a {\em matching} in $G$ if no two edges in $M$ share an endvertex. The maximum cardinality of a matching in $G$ is the {\em matching number}, $\alpha'(G)$, of $G$. Let $[k]=\{1,\ldots,k\}$. A {\em (proper) $k$-coloring} of $G$ is a mapping $c:V(G)\to [k]$ such that $c^{-1}(i)$ is an independent set in $G$ for all $i\in [k]$. A graph $G$ is {\em $k$-colorable} if it admits a proper $k$-coloring.
 
Given a graph $G$ and $\tau\in \{\alpha,\beta,\alpha',\gamma,\iota,\rho_2\}$, $\tau(G)$-set represents an independent set, a vertex cover, a matching, a dominating set, an isolating set, or a $2$-packing, respectively, of cardinality $\tau(G)$.

The \textit{Cartesian product} of graphs $G$ and
$H$ is the graph, denoted as $G\, \Box\, H$, with vertex set $V (G)\times
V(H)$, where two vertices $(g,h)$ and $(g_1,h_1)$ are adjacent if and only if $g=g_1$ and $hh_1 \in E(H)$, or $gg_1 \in E(G)$ and $h=h_1$.
 The {\em lexicographic product} of $G$ and $H$ has vertex set $V(G) \times V(H)$ and two vertices $(g,h)$ and $(g_1,h_1)$ are adjacent if $gg_1 \in E(G)$, or if $g=g_1$ and $hh_1 \in E(H)$.  This product is denoted $G \circ H$.

For a fixed $h\in V(H)$, the subgraph $G^h$ of vertices in $G\,\Box\, H$ (respectively $G\circ H$) whose second coordinate is $h$ is the {\em $G$-fiber with respect to $h$}. In a similar way the $H$-fiber with respect to a vertex $g\in V(G)$ is defined in both graph products, and we denote it by $^{g}\! H$.
Note that every $G$-fiber is isomorphic to $G$ and every $H$-fiber is isomorphic to $H$.
Let $A\subseteq V(G)\times V(H)$ be a subset of vertices in $G\,\Box\, H$ (respectively, $G\circ H$). By $p_G(A)$ we denote the set $\{g\in V(G):\, (g,h)\in A\}$ and $p_H(A)$ is the set $\{h\in V(H):\, (g,h)\in A\}$.

We finish the section by presenting some straightforward observations concerning the isolation number.

If $S \subseteq V(G)$, then $\gamma_G(S)$ is the cardinality of a smallest $D \subseteq V(G)$ such that $S \subseteq N[D]$.
The {\em independence-domination} number of a graph $G$, as introduced by Aharoni and Szab\'{o}~\cite{aha-2009}, is the invariant defined by 
\[\gamma^{i}(G)=\max \{\gamma_G(S)\,:\, S \text{ is an independent set of } G \}\,.\]

It is clear that $\rho_2(G) \leq \gamma^{i}(G) \leq \gamma(G)$ for any graph $G$. On the other hand, $\iota(G)$ is, in general, not comparable with $\rho_2(G)$ and $\gamma^i(G)$. For instance, consider the subdivision graph $G=S(K_{1,n})$ of the star $K_{1,n}$, obtained by subdividing each edge of the star by one vertex. Clearly, $\iota(G)=1$ and $\rho_2(G)=n=\gamma^i(G)$. On the other hand, in the Cartesian product $K_n\,\Box\, K_n$, we have $\rho_2(K_n\,\Box\, K_n)=1$ whereas $\iota(K_n\,\Box\, K_n)=n-1$. In  addition, consider an example presented in~\cite{aha-2009}. Given a positive integer $n$, let $G$ be the line graph of the hypergraph of all $n$-subsets of the set on $n^2$ vertices. Note that $\gamma^i(G)=1$, yet $\iota(G)=n-1$.

In Section~\ref{sec:Cartesian}, $2$-packing will be used to obtain a lower bound on the isolation number of the  Cartesian product of two graphs. Similarly, $\gamma^i(G)$ is involved in the following lower bound on the isolation number of a graph.

\begin{proposition}
    If $G$ is any graph, then $\iota(G) \ge \gamma(G)-\gamma^{i}(G)$.  This bound is sharp.
\end{proposition}
\begin{proof}
  Let $I$ be an $\iota(G)$-set and let $A=V(G)-N[I]$.  Let $B$ be a smallest subset of $V(G)$ such that $A \subseteq N[B]$.  By the definition of the independence-domination number, $|B| \leq \gamma^{i}(G)$.  Furthermore, $B$ and $I$ are disjoint and $B \cup I$ is a dominating set of $G$.  This implies that
  \[\gamma(G) \leq |B|+|I| \leq \gamma^{i}(G) +\iota(G)  \,,\] or equivalently $\iota(G) \geq \gamma(G)-\gamma^{i}(G)$.  The bound is sharp as can be seen for $G=K_{m,n}$ such that $2 \le m \le n$.
\end{proof}

Note that $\gamma(G)\le \alpha'(G)$ in any graph $G$ without isolated vertices~\cite{bc-1979}. The bound $\iota(G)\le \alpha'(G)$, which immediately follows, can be improved by using the {\em saturation number} of $G$, which is the minimum cardinality of a maximal matching in $G$, and is denoted by $s(G)$. Indeed, let $A$ be a maximal matching of minimum cardinality in $G$, and let $X$ be a set of vertices with $|X|=|A|$ such that each edge in $A$ has exactly one endvertex in $X$. Since $A$ is a maximal matching in $G$, every edge in $E(G)-A$ has a common endvertex with an edge in $A$, therefore $G-N[X]$ has no edges. Thus, $\iota(G)\le |X|=|A| = s(G)$.

\section{Cartesian products of graphs}
\label{sec:Cartesian}

Isolating sets in Cartesian products of graphs have been considered in~\cite{ch-2017}, where lower and upper bounds on the isolation numbers of Cartesian products of two cycles/paths have been determined. We expand this investigation by considering isolation in general Cartesian products of graphs, and some important families of Cartesian products such as hypercubes.

In terms of domination, there exists a trivial upper bound for the domination number of the Cartesian product of two graphs (see~\cite[Proposition 18.2]{HHH3}): $$\gamma(G \,\Box\, H) \leq \min{\{\gamma(G)n(H),\gamma(H)n(G)\}}.$$
This bound cannot extend trivially to the isolation number of a Cartesian product, since for an $\iota(G)$-set $D$, the set $\{(g,h); g \in D, h \in V(H)\}$ is not an isolating set as soon as $H$ has at least one edge and $D$ is not a dominating set of $G$. The issue is not merely that the proof strategy is unusable, the bound is also invalid. For $n \geq 3$, consider the subdivision graph $S(K_{1,n})$ of the star $K_{1,n}$, obtained by subdividing each edge of the star by one vertex. Then $\io(S(K_{1,n}))=1$ and $$\io(S(K_{1,n}) \,\Box\, K_2)=n > \min{\{n(S(K_{1,n}))\io(K_2),n(K_2)\io(S(K_{1,n}))\}}=2.$$ The example also implies that the difference between  $\io(G \,\Box\, H)$ and the minimum of 
$n(G)\io(H)$ and $n(H)\io(G)$  can be arbitrary large. However, one can easily find an upper bound for the isolation number of the Cartesian product of two graphs in terms of the independence number, the domination number and the isolation number of the factor graphs.

\begin{proposition}\label{prop:trivialUpper}
    If $G$ and $H$ are graphs, then $$\io(G\Box H) \leq \min{\{\alpha(G)\io(H)+ \beta(G)\gamma(H), \alpha(H)\io(G)+ \beta(H)\gamma(G)\}}.$$
\end{proposition}
\begin{proof}
    Let $A$ be an $\alpha(G)$-set, $I$ an $\io(H)$-set, and $D$ a $\gamma(H)$-set. Then $$S=\{(g,h)\,:\,g \in A,h\in I\} \cup \{(g,h)\,:\, g\in V(G)-A, h\in D\}$$ is an isolating set of $G \,\Box\, H$.  Reversing the roles of $G$ and $H$ completes the proof.
\end{proof}

In what follows, we will improve the bound from Proposition~\ref{prop:trivialUpper} by using  several definitions. If $G$ is a graph, then $\omega(G)$ denotes its {\em clique number}, which is the order of a largest complete subgraph of $G$. Next, if  $A$ is an $\iota(G)$-set, then let $L_A=V(G)-N[A]$.  That is, $L_A$ consists of vertices that are not dominated by $A$, and so $L_A$ is an independent set in $G$. Now, the {\em isolation graph} of $G$ is the graph $I(G)$ whose vertices are the $\iota(G)$-sets, and $A,B\in V(I(G))$ are adjacent whenever $L_A\cap L_B=\emptyset$.
Finally, for a graph $G$ and a positive integer $k$, the maximum cardinality of an induced $k$-colorable subgraph in $G$ will be denoted by $\alpha_k(G)$.

\begin{theorem}\label{thm:uperBoundProduct}
If $G$ and $H$ are graphs, then
$$\iota(G\,\Box\, H)\le \min\{\alpha_k(G)\iota(H)+(n(G)-\alpha_k(G))\gamma(H),\alpha_\ell(H)\iota(G)+(n(H)-\alpha_\ell(H))\gamma(G)\},$$
where $k=\omega(I(H))$ and $\ell=\omega(I(G))$. The bound is sharp.
\end{theorem}
\begin{proof}

Let $H'=I(H)$ be the isolation graph of $H$. Let $k=\omega(H')$ and let $\{A_1,\ldots , A_k\} \subseteq V(H')$ be a set that induces a clique in $H'$. Hence, for any $i \in [k]$, $A_i$ is an $\iota(H)$-set. Let $G' \subseteq G$ be the largest induced $k$-colorable subgraph of $G$. Furthermore, let $C_1,\ldots , C_k$ be the color classes of $V(G')$ obtained from an arbitrary $k$-coloring $c:V(G') \to [k]$. Let $D_H$ be an arbitrary $\gamma(H)$-set  and let $D_1=\{(g,h):\, g \in V(G)-V(G'),h \in D_H\}$. Now, let $D=D_1 \cup \bigcup_{i=1}^k 
 (C_i\times A_i)$. Since $|D|=\alpha_k(G)\io(H)+(n(G)-\alpha_k(G))\gamma(H)$, it remains to show that $D$ is an isolating set of $G \,\Box\, H$.
Clearly, $D_1$ dominates $(V(G)-V(G')) \times V(H)$. Furthermore, for any $i \in [k]$, $C_i \times A_i$ is an isolating set of the subgraph of $G \,\Box\, H$ induced by 
$C_i \times V(H)$. Let $I_i$ be the set of vertices in this subgraph that are not dominated by $C_i \times A_i$. By the definition of the isolation graph of $H$ it follows that for any $i \neq j$, $L_{A_i} \cap L_{A_j}=\emptyset$ (since $A_i$ and $A_j$ are adjacent in $H'$), which implies that for any $i,j \in [k]$, $I_i \cup I_j$ is an independent set in $G \,\Box\, H$. Consequently, $\cup_{i=1}^{k}I_i$ is an independent set in $G \,\Box\, H$ and thus $D$ is an isolating set of $G \,\Box\, H$, as desired. By reversing the roles of $G$ and $H$, the statement follows.

Note that $\iota(K_3)=1$ and $I(K_3)=K_3$, thus $\omega(I(K_3))=3$. Clearly, $\alpha_3(K_2)=2$. Thus, $\iota(K_3\,\Box\, K_2)=2=\alpha_3(K_2)\iota(K_3)$, which shows that the bound is sharp.
\end{proof}

The bound in Theorem~\ref{thm:uperBoundProduct} is sharp for an infinite family of pairs of graphs. 
Notably, take $G=K_n$ and $H=C_n$, where $n\ge 5$.  Note that $I(K_n)=K_n$, and so $\omega(I(K_n))=n$. Clearly, $\alpha_n(C_n)=n$, and so $\alpha_n(C_n)\iota(K_n)=n$.  On the other hand,  $\alpha_k(K_n)\iota(C_n)+(n(K_n)-\alpha_k(K_n))\gamma(C_n)\ge n\iota(C_n)>n$.  To prove that the bound in the theorem above is attained we need to prove that $\iota(K_n\,\Box\, C_n)=n$.

\begin{proposition}
\label{prp:KnCn}
 If $n\ge 5$, then $\iota(K_n\,\Box\, C_n)=n$.   
\end{proposition}
\begin{proof}
   Let $V(K_n)=\{u_1,\ldots, u_n\}$ and $V(C_n)=\{v_1,\ldots, v_n\}$. First, note that $\gamma(K_n\,\Box\, C_n)=n$. Suppose that $\iota(K_n\,\Box\, C_n)<n$. Then, there exists an $\iota$-set $A$ of $K_n\,\Box\, C_n$, where $|A|\le n-1$, such that $A$ is not a dominating set. By symmetry we may assume that $(u_1,v_1)\in V(K_n\,\Box\, C_n)-N[A]$. Since $A$ is an isolating set, vertices $(u_2,v_1),(u_3,v_1),\ldots,(u_n,v_1)$ are dominated by $A$, but none of them belongs to $A$. This means that in each of the pairs $\{(u_i,v_2),(u_i,v_n)\}$, for all $i\in\{2,\ldots, n\}$,  exactly one vertex is in $A$. Therefore, $|A|=n-1$ and $A$ is determined up to the selection of vertices from the described pairs. However, $(u_1,v_3)$ and $(u_1,v_4)$ are adjacent vertices neither of which is dominated by $A$, which is a contradiction.  
\end{proof}

 From Proposition~\ref{prp:KnCn} and the discussion preceding it we infer that 
$$\iota(K_n\,\Box\, C_n)=n=\alpha_{\omega(I(K_n))}(C_n)\iota(K_n),$$
which gives an infinite family of pairs of graphs that attain the upper bound in Theorem~\ref{thm:uperBoundProduct}.

Clearly, $\alpha_k(G)\le n(G)$ in any graph $G$, and if $G$ is $k$-colorable, then $\alpha_k(G)=n(G)$. Based on this fact, the following result follows immediately from Theorem~\ref{thm:uperBoundProduct}.
\begin{corollary}
 If $G$ and $H$ are graphs such that $G$ is $\omega(I(H))$-colorable, then
$\iota(G\,\Box\, H)\le  n(G)\iota(H)$.
\end{corollary}

An alternative upper bound can be obtained by using the vertex cover numbers of both factors.

\begin{theorem}\label{thm:BetaBound}
If $G$ and $H$ are graphs having no isolated vertices, then
$$\iota(G\,\Box\, H)\le \beta(G)\beta(H).$$ Moreover, the bound is sharp.
\end{theorem}
\begin{proof}
    Let $A_G$ be an $\alpha(G)$-set and $A_H$ an $\alpha(H)$-set. Denote $D_G=V(G)-A_G$ and $D_H=V(H)-A_H$. Since $G$ has no isolated vertices, every vertex in $A_G$ has at least one neighbor in $D_G$, which implies that $D_G$ is a dominating set of $G$, and similarly, we derive that $D_H$ is a dominating set of $H$. Now consider the set $D=D_G \times D_H$. Then $D$ dominates all vertices in $V(G  \,\Box\, H)-(A_G \times A_H)$. Since $A_G \times A_H$ is an independent set of $G \,\Box\, H$, the set $D$ is an isolating set of $G \,\Box\, H$. Hence $\io(G \,\Box\, H)\leq (n(G)-\alpha(G))(n(H)-\alpha(H))=\beta(G)\beta(H)$. 
    For the sharpness, $\io(P_5 \,\Box\, P_5)=4=\beta(P_5)\beta(P_5)$. There is also an infinite family that achieves the bound, for $m,n\geq 2$ we get, $\io(K_{1,n} \,\Box\, K_{1,m})=1=\beta(K_{1,n})\beta(K_{1,m})$.
\end{proof}

Note that bounds from Theorem~\ref{thm:uperBoundProduct} and Theorem~\ref{thm:BetaBound} are not comparable. 
Since for $n\ge 3$, $I(K_{1,n})=K_{1,n}$, we get $\omega(I(K_{1,n}))=2$. 
Clearly, $\alpha_2(K_{1,n})=n+1$. 
For $3 \leq n \leq m$ and graphs $G=K_{1,n}$ and $H=K_{1,m}$ it holds that $$\min\{\alpha_2(G)\iota(H)+(n(G)-\alpha_2(G))\gamma(H),\alpha_2(H)\iota(G)+(n(H)-\alpha_2(H))\gamma(G)\}=n+1$$ and $\beta(G)\beta(H)=1$ so the bound from Theorem~\ref{thm:BetaBound} is better. On the other hand, for $n\ge 3$ we have $I(K_n)=K_n$, and so $\omega(I(K_n))=n$. Clearly, if $m\ge n$, we have $\alpha_m(K_n)=n=\alpha_n(K_m)$.  For $3 \leq n\leq m$ and graphs $G=K_n$ and $H=K_m$, it holds 
$$\min\{\alpha_m(G)\iota(H)+(n(G)-\alpha_m(G))\gamma(H),\alpha_n(H)\iota(G)+(n(H)-\alpha_n(H))\gamma(G)\}=n$$ and $\beta(K_n)\beta(K_m)=(n-1)(m-1)$, which implies that the bound from Theorem~\ref{thm:uperBoundProduct} is better.

It is generally harder to obtain some good lower bounds on the isolation number of a graph. The following result provides a lower bound on the isolation number of the Cartesian product of two graphs.

\begin{theorem}
If $G$ and $H$ are graphs, then
$$\iota(G\,\Box\, H)\ge \max\{\rho_2(G)\iota(H),\rho_2(H)\iota(G)\}.$$
\end{theorem}
\begin{proof}
By symmetry, it suffices to prove that $\iota(G\,\Box\, H)\ge \rho_2(G)\iota(H)$.  Suppose to the contrary that $\iota(G\,\Box\, H)< \rho_2(G)\iota(H)$. Let $D$ be an $\iota(G\,\Box\, H)$-set and let $A$ be a $\rho_2(G)$-set. Thus, there exists a vertex $x\in A$ such that $$|(N_G[x]\times V(H))\cap D|<\iota(H).$$
Let $D'=p_H\big((N_G[x]\times V(H))\cap D\big)$, and note that $|D'|<\iota(H)$. Therefore, there exist two vertices $h,h'\in V(H)-N[D']$ that are adjacent in $H$. Hence $(x,h)(x,h')\in E(G\,\Box\, H)$, and no vertex in $V(G\,\Box\, H)\cap D$ dominates $(x,h)$ or $(x,h')$. Thus, $D$ is not an isolating set in $G\,\Box\, H$, which is a contradiction.
\end{proof}

Given a graph $G$, the Cartesian product $G\,\Box\, K_2$ is the {\em prism} of $G$. Next, we present a lower and an upper bound on the isolation number of an arbitrary prism. Moreover, when $G$ is bipartite, the lower bound is the exact value of the isolation number of the prism of $G$.

\begin{theorem}
    \label{thm:prism}
If $G$ is a graph, then $$\gamma(G)\le \iota(G\,\Box\, K_2)\le \gamma(G\,\Box\, K_2).$$ 
In addition, if $G$ is bipartite, then $\iota(G\,\Box\, K_2)=\gamma(G)$.
\end{theorem}
\begin{proof}
Set $V(K_2)=[2]$. Let $D$ be an $\iota(G\,\Box\, K_2)$-set and $D'=p_G(D)$ be its projection to $G$. Note that $D'$ is a dominating set of $G$ for otherwise, if there exists $x\notin N_G[D']$, then $(x,1)$ and $(x,2)$ are not dominated by $D$ in $G\,\Box\, K_2$, a contradiction with $D$ being an isolating set. Thus, $\iota(G\,\Box\, K_2)=|D|\ge |D'|\ge \gamma(G)$. This settles the first inequality, while the second one is trivial.

Now, let $G$ be a bipartite graph with the bipartition $V(G)=A_1\cup A_2$ where $A_1$ and $A_2$ are independent sets. Let $D$ be a $\gamma(G)$-set, and let $D_i=D\cap A_i$ for $i\in [2]$. We claim that the set $A=(D_1\times \{1\})\cup (D_2\times\{2\})$ is an isolating set of $G\,\Box\, K_2$. 
Let $v \in A_2$.  If $v \in D_2$, then $(v,1)$ is dominated by $D_2 \times \{2\}$.  If $v \in A_2-D_2$, then $(v,1)$ is dominated by $D_1 \times \{1\}$ since $D$ is a dominating set of $G$.  Hence, every vertex in $A_2 \times \{1\}$ is dominated by $A$.  Similarly, $A$ dominates $A_1 \times \{2\}$.  It follows that
$V(G\, \Box\,H) -N[A] \subseteq (A_1 \times \{1\}) \cup (A_2 \times \{2\})$.  Since  
$(A_1 \times \{1\}) \cup (A_2 \times \{2\})$ is independent, we conclude that $A$ is an isolating set. Since $|A|=|D|=\gamma(G)$, we infer by using also the first inequality of the theorem: 
$$\gamma(G)\le \iota(G\,\Box\, K_2)\le |A|=|D|=\gamma(G).$$
The proof is complete.
\end{proof}

For the class of hypercubes $Q_n$, Theorem~\ref{thm:prism} implies that $$\iota(Q_{n+1})=\gamma(Q_n)$$
holds for any $n\in \mathbb{N}$. The problem of determining the exact values of the domination numbers of hypercubes is a long-standing open problem; see a recent paper~\cite{bkr-2024} and the references therein. For this reason, obtaining exact values of $\iota(Q_n)$ for all $n$ will also be hard.

\section{Lexicographic products and graphs with $\iota=\gamma$}
\label{sec:lex}

When starting the study of the isolation number, one may wonder if the natural upper bound $\iota(G)\leq\gamma(G)=k$ can be attained for all $k\in \mathbb{N}$.  Clearly, in many graphs $\gamma(G)$ can be arbitrarily larger than $\iota(G)$, just take subdivisions of stars, where $\iota(S(K_{1,n}))=1<n=\gamma(S(K_{1,n}))$.
In this section, we present a large family of graphs in which both invariants coincide. For this purpose we will use yet another graph product---the lexicographic product.  

We start by recalling the formula about the domination number of the lexicographic product of two graphs.

\begin{theorem}\hskip-0.5pt {\rm \cite{spt}} 
\label{thm:gammalex}
If $G$ is a nontrivial connected graph and $H$ is a connected graph with $\gamma(H)\ge 2$, then $\gamma(G \circ H)=\gamma_t(G)$.
\end{theorem}

We now show that the isolation number equals the domination number in all lexicographic products of two nontrivial connected graphs where the isolation number of the second factor is at least $2$.

\begin{theorem} \label{thm:lexico}
If $G$ and $H$ are nontrivial connected graphs such that $\iota(H)\geq 2$, then 
\[\iota(G \circ H)=\gamma(G \circ H)=\gamma_t(G)\,.\]
\end{theorem}

\begin{proof}
By definition, $\iota(G \circ H)\leq \gamma(G \circ H)$ and by Theorem~\ref{thm:gammalex}, $\gamma(G \circ H)=\gamma_t(G)$.  Thus, it remains to show that 
$\gamma_t(G) \leq \iota(G \circ H)$.

We claim that there exists an $\iota(G \circ H)$-set such that every $H$-fiber contains at most one vertex of that set. Let $A$ be an $\iota(G \circ H)$-set with the least number of $H$-fibers that contain more than one vertex in $A$.  Suppose that for some $x \in V(G)$, the $H$-fiber $^{x}\! H$ contains at least two vertices, say $(x,h)$ and $(x,k)$, from $A$.  If there exists $w \in N_G(x)$ such that $|V(^{w}\! H) \cap A| \ge 1$, then $A-\{(x,h)\}$ is a smaller isolating set of $G\circ H$, which is a contradiction.  Thus, if $y$ is an arbitrary neighbor of $x$ in $G$, then $V(^{y}\! H) \cap A=\emptyset$. Now, let 
$A'=(A-\{(x,k)\}) \cup \{(y,k)\}$. Note that $N[A] \subseteq N[A']$, and so $A'$ is also an $\iota(G\circ H)$-set, yet there are fewer $H$-fibers that contain more than one vertex in $A'$ than there were such $H$-fibers with respect to $A$. This is a contradiction, which proves the claim.  

Thus, we may assume that no $H$-fiber contains more than one vertex of an $\iota(G\circ H)$-set $A$.  Let $D=p_G(A)$.  We claim that $D$ is a total dominating set of $G$. First, suppose that there exists $g \in V(G)-N[D]$.  Since $H$ contains at least one edge, say $h_1h_2$, the edge $(g,h_1)(g,h_2)$ is in the graph $G \circ H -N[A]$, which is a contradiction. Hence $D$ is a dominating set of $G$.  Suppose now, that there exists a vertex $d \in D$ that has no neighbors in $D$, and let $A \cap V(^{d}\! H)=\{(d,h')\}$. Since $d$ has no neighbors in $D$,  $(d,h')$ is the only vertex in $A$ that dominates vertices in $^{d}\! H$.  Since $\iota(H) \ge 2$, it follows that $A$ is not an isolating set of $G \circ H$.  This is a contradiction, and we conclude that $D$ is a total dominating set of $G$.  Therefore, \\
$\gamma(G \circ H)=\gamma_t(G) \le |D|=|A|=\iota(G \circ H)\,.$
\end{proof}

By fixing a graph $H$ such that $\iota(H) \ge 2$ we can use Theorem~\ref{thm:lexico} to construct an infinite family of graphs that have equal domination and isolation numbers.  In fact, for each $k\ge 2$ we now present an infinite family $\{F_i\}$ such that $\iota(F_n)=\gamma(F_n)= k$ for every positive integer $n$.  Let $H_1=C_5$ and for each $i \ge 2$ let $H_i=H_{i-1} \vee H_{i-1}$, where $\vee$ stands for the join of two graphs (that is, $H_i$ is obtained from the disjoint union of two copies of $H_{i-1}$ by adding an edge between each vertex of one copy and each vertex in another copy of $H_{i-1}$). 
It is easy to see that 
$\iota(H_i)=2$, for every $i$.  Now let $G$ be a graph that has total domination number $k$.  For each $n$, let 
$F_n=G \circ H_n$.  By Theorem~\ref{thm:lexico}, $\iota(F_n)=\gamma(F_n)=k$, for every $n$.

\section{Sierpi\'{n}ski graphs}
\label{sec:Sierp}

Gravier et al.\ in~\cite{gkp-2011} introduced the $t$-th generalized Sierpi\'{n}ski graphs as follows. If $G$=(V,E) is an arbitrary graph (called the base graph of the generalized Sierpi\'{n}ski graph), then $S_G^t$ is the graph with vertex set $V^t$ and edge set defined as follows. Vertices $u=u_1\ldots u_t$ and $v=v_1\ldots v_t$ are adjacent in $S_G^t$ if and only if there exists $i \in [t]$ such that
\begin{enumerate}
    \item[(i)] $u_j = v_j$ if $j < i$;
    \item[(ii)] $u_i \neq  v_i$ and $u_iv_i \in E$; 
    \item[(iii)] $u_j = v_i$ and $v_j = u_i$ if $j > i$.
\end{enumerate}

If $uv \in E(S_G^t)$, then there is an edge $xy \in E(G)$ and a word $w$ (which can be the empty word) such that $u = wxyy\ldots y$ and $v = wyxx\ldots x$. Moreover, $S_G^t$ can be constructed
recursively from $G$ in the following way: $S_G^1 = G$ and, for $t \geq 2$, $S_G^t$ is obtained from $|V(G)|$ disjoint copies of $S_G^{t-1}$ by replacing each vertex of $G$ by a copy of $S_G^{t-1}$ and then by adding the letter $x$ at the beginning of each label of the vertices belonging to the copy of $S_G^{t-1}$ corresponding to $x$. Then for every edge $xy \in E(G)$, we add an edge between the vertex $xyy \ldots y$ and the vertex $yxx \ldots x$. Vertices of the form $xx \ldots x$ are called {\em extreme vertices} of $S_G^t$. For $i \in V(G)$, the subgraph of $S_G^t$ induced by the set $\{iw; w \in V(S_G^{t-1})\}$ is isomorphic to $S_G^{t-1}$ and is denoted by $iS_G^{t-1}$. Notice that for any graph $G$ of order $n$ and any integer $t \geq 2$, $S_G^t$ has $n$ extreme vertices and, if $x$ has degree $k$ in $G$, then the extreme vertex $x\ldots x$ of $S_G^t$ also has degree $k$. Moreover, the degree of the vertex $yxx\ldots x$ in $S_G^t$ is $k+1$, since $yx\ldots x$ has exactly one neighbor in $S_G^t$ that is outside $yS_G^{t-1}$.
The set of extreme vertices of $S_G^t$ will be denoted by $Ex(S_G^t)$. For a set $D \subseteq V(S_G^t)$, $iD$ denotes the vertices in $D \cap V(iS_G^{t-1})$.

To obtain the upper bound for the isolation number of generalized Sierpi\'{n}ski graphs we first need the following notations. With ${\cal{I}}(G)$ we denote the set of all isolating sets of $G$ of cardinality at most $\gamma(G)$, i.e. $${\cal{I}}(G)=\{D; D \textrm{ is isolating set and } |D| \leq \gamma(G)\}.$$ Moreover, let 
$$\xi(S_G^2)=\min{\{|D|:\, D\in {\cal{I}}(S_G^2), \textrm{ where }  ii,jj\in V(S_G^2)-N[D]\textrm{ implies }ij\notin E(G)\}}.$$ The set $D \subseteq V(S_G^2)$ of cardinality $\xi(S_G^2)$ with the property that no two extreme vertices $ii,jj$ with $ij\in E(G)$ belong to $V(S_G^2)-N[D]$ is called a $\xi(S_G^2)$-set.

\begin{theorem}\label{thm:upperBound}
    If $G$ is a graph and $t\geq 2$, then $\io(S_G^t) \leq \xi(S_G^2)\cdot n(G)^{t-2}$.
\end{theorem}
\begin{proof}
    Let $G$ be a graph of order $n$ with $V(G)=[n]$ and let $D^2$ be a $\xi(S_G^2)$-set. For $t\geq 3$ we inductively define $D^t=\cup_{i=1}^{n} iD^{t-1}$, where $iD^{t-1} = \{iw:\, w \in D^{t-1}\}$ for any $i\in V(G)$.
    We will prove that $D^t$ is an isolating set of $S_G^t$ with the property that no two extreme vertices $i^t,j^t$ with $ij\in E(G)$ belong to $V(S_G^t)-N[D^t]$, assuming that $D^{t-1}$ satisfied such a property. Note that the basis of induction is fulfilled, since $D^2$ has such a property by the definition of the invariant $\xi(S_G^2)$.
 
     Note that $S_G^t$ is obtained from $n$ copies of $S_G^{t-1}$, each copy corresponding to a vertex in $V(G)$, by adding the edges between vertices $ij^{t-1}$ (that is, the extreme vertex $j^{t-1}$ of the copy of $S_G^{t-1}$ that corresponds to the vertex $i\in V(G)$) and $ji^{t-1}$ (which is the extreme vertex $i^{t-1}$ of the copy of $S_G^{t-1}$ that corresponds to the vertex $j\in V(G)$) if $ij\in E(G)$. Note that the only vertices in a copy of $S_G^{t-1}$ that have a neighbor in another copy of $S_G^{t-1}$ are the extreme vertices of that copy. Therefore, for $i\in V(G)$, in the subgraph $iS_G^{t-1}$ any two vertices that are not dominated by $iD^{t-1}$ are not adjacent, because $D^{t-1}$ is an isolating set in $S_G^{t-1}$ and $D^t\cap V(iS_G^{t-1})=iD^{t-1}$. Hence, it remains to check that two vertices that are not dominated by $D^t$ and belong to different copies of $S_G^{t-1}$ are also not adjacent. As noted above, $ij^{t-1}$ is adjacent to $ji^{t-1}$ if and only if $ij\in E(G)$. Assume that $ij\in E(G)$ for two vertices $i$ and $j$ in $G$. By the induction hypothesis, at most one of the vertices $i^{t-1}$ and $j^{t-1}$ is not dominated by $D^{t-1}$. Therefore, at most one vertex among $ij^{t-1}$ and $ji^{t-1}$ is not dominated by $D^t$. Thus, if  $ij^{t-1}$ is not dominated by $D^t$, then all its neighbors in $iS_G^{t-1}$ are dominated by $D^t$ and also $ji^{t-1}$ is dominated by $D^t$.
We conclude that $D^t$ is an isolating set of $S_G^t$ with the property that no two extreme vertices $i^t,j^t$ with $ij\in E(G)$ belong to $V(S_G^t)-N[D^t]$.

Note that $|D^t|=|D^2|n(G)^{t-2}=\xi(S_G^2)n(G)^{t-2}$. The proof is complete.
\end{proof}

Note that $\xi(S_G^2) \geq \io(G)$ and if $G$ is a graph with the property that there exists an $\io(S_G^2)$-set $D$ so that among any two extreme vertices $ii,jj$ of $S_G^2$, where $ij\in E(G)$, at most one is not dominated by $D$, then $\xi(S_G^2)=\io(S_G^2)$. Then we get the following. 

\begin{corollary}\label{cor:upperBound}
    If $G$ is a graph with $\xi(S_G^2)=\io(S_G^2)$ and $t\ge 2$, then $$\io(S_G^t)\leq \io(S_G^2)n(G)^{t-2}.$$ 
\end{corollary}

For $X\subseteq V(G)$, let $G|X$ denote the graph $G$ in which vertices from $X$ are considered as being already dominated. Then it clearly holds $\io(G|X) \leq \io(G)$. Now we are ready to prove a lower bound.

\begin{theorem}\label{thm:lower}
    If $G$ is a graph and $t\geq 2$, then $\io(S_G^t) \geq \io(S_G^2|Ex(S_G^2))\cdot n(G)^{t-2}$.
\end{theorem}

\begin{proof}
    Let $D^t$ be an $\io(S_G^t)$-set. Note that $S_G^t$ contains $n(G)^{t-2}$ disjoint copies of $S_G^2$, where only extreme vertices in these copies of $S_G^2$ (possibly) have neighbors in $S_G^t$ that are outside the copy. Let $G'$ be an arbitrary copy of $S_G^2$ in $S_G^t$. Since $D^t \cap V(G')$ is an isolating set of $S_G^2|Ex(S_G^2)$, it follows that $|D^t \cap V(G')| \geq \io(S_G^2|Ex(S_G^2))$ and thus $|D^t| \geq \io(S_G^2|Ex(S_G^2)) \cdot n(G)^{t-2}$.
\end{proof}

In the next results we will see that bounds from Theorem~\ref{thm:upperBound} and Theorem~\ref{thm:lower} are sharp. In fact, the bounds lead to exact values of the isolation numbers in Sierpi\'{n}ski graphs whose base graphs are complete graphs. 

\begin{corollary}\label{prop:Complete}
    If $t\geq 2$ and $n\ge 3  $, then $\io(S_{K_n}^t)=(n-1)\cdot n^{t-2}$.
\end{corollary}   

\begin{proof}
    Let $V(K_n)=[n]$. Clearly $\io(S_{K_n}^2)=n-1$. Since $D=\{1n,2n,\ldots,(n-1)n\}$ is an $\io(S_{K_n}^2)$-set that is also $\xi(S_{K_n}^2)$-set, it follows that $\xi(S_{K_n}^2)=\io(S_{K_n}^2)=n-1$. Thus Theorem~\ref{thm:upperBound} implies that $\io(S_{K_n}^t) \leq (n-1)n^{t-2}$.

    For the converse, let $D$ be an $\io(S_{K_n}^2|Ex(S_{K_n}^2))$-set. Suppose that $\io(S_{K_n}^2|Ex(S_{K_n}^2)) \leq n-2$. Hence, there are at least two copies of $K_n$ in $S_{K_n}^2$, say $iS_{K_n}^1$ and $jS_{K_n}^1$, that contain no vertices from $D$. Since $V(iS_{K_n}^1) \cap D=\emptyset$ and $ii$ has no neighbors in $V(S_{K_n}^2)-V(iS_{K_n}^1)$, all $n-1$ non-extreme vertices of $iS_{K_{n}}^1$ are dominated from outside this copy and the extreme vertex $ii$ is isolated by $D$.  Since $V(jS_{K_n}^1) \cap D =\emptyset$, two adjacent vertices $ii$ and $ij$ are not dominated by $D$, a contradiction. Hence $\io(S_{K_n}^2|Ex(S_{K_n}^2)) \geq n-1$. Since $\io(S_{K_n}^2) = n-1 \geq \io(S_{K_n}^2|Ex(S_{K_n}^2)) \geq n-1$, we deduce that $\io(S_{K_n}^2|Ex(S_{K_n}^2)) = n-1$. Thus Theorem~\ref{thm:lower} implies that $\io(S_{K_n}^t) \geq (n-1)n^{t-2}$.
\end{proof}

\begin{figure}[ht!]
\begin{center}
\begin{tikzpicture}[scale=0.8  ,style=thick,x=1cm,y=1cm]
\def\vr{3pt}

\begin{scope}[xshift=-10cm, yshift=3cm] 
\coordinate(x1) at (0,0);
\coordinate(x2) at (0,1);
\coordinate(x3) at (1,1);
\coordinate(x4) at (1,0);
\draw (x1) -- (x2) -- (x3) -- (x4) -- (x1);

\foreach \i in {1,...,4}
{ 
\draw(x\i)[fill=white] circle(\vr);
}
\draw(x1)[fill=black] circle(\vr);

\draw[below] (x1)++(0,-0.1) node {$1$};
\draw[above] (x2)++(0,0) node {$2$};
\draw[above] (x3)++(0,0) node {$3$};
\draw[below] (x4)++(0,-0.1) node {$4$};
\end{scope}

\begin{scope}[xshift=-6cm, yshift=2cm] 
\coordinate(x1) at (0,0);
\coordinate(x2) at (1,0);
\coordinate(x3) at (2,0);
\coordinate(x4) at (3,0);
\coordinate(x5) at (0,1);
\coordinate(x6) at (1,1);
\coordinate(x7) at (2,1);
\coordinate(x8) at (3,1);
\coordinate(x9) at (0,2);
\coordinate(x10) at (1,2);
\coordinate(x11) at (2,2);
\coordinate(x12) at (3,2);
\coordinate(x13) at (0,3);
\coordinate(x14) at (1,3);
\coordinate(x15) at (2,3);
\coordinate(x16) at (3,3);
\draw (x1) -- (x2) -- (x6) -- (x5) -- (x1);
\draw (x2) -- (x3) -- (x4) -- (x8) -- (x7) -- (x3);
\draw (x8) -- (x12) -- (x16) -- (x15) -- (x11) -- (x12);
\draw (x15) -- (x14) -- (x13) -- (x9) -- (x10) -- (x14);
\draw (x5) -- (x9);
\foreach \i in {1,...,16}
{ 
\draw(x\i)[fill=white] circle(\vr);
}
\draw(x14)[fill=black] circle(\vr);
\draw(x8)[fill=black] circle(\vr);
\draw(x1)[fill=black] circle(\vr);


\draw[below] (x1)++(0,-0.1) node {$11$};
\draw[below] (x2)++(0,-0.1) node {$14$};
\draw[below] (x3)++(0,-0.1) node {$41$};
\draw[below] (x4)++(0,-0.1) node {$44$};
\draw[above] (x13)++(0,0.1) node {$22$};
\draw[above] (x14)++(0,0.1) node {$23$};
\draw[above] (x15)++(0,0.1) node {$32$};
\draw[above] (x16)++(0,0.1) node {$33$};
\draw[left] (x5)++(0,0) node {$12$};
\draw[left] (x9)++(0,0) node {$21$};
\draw[right] (x8)++(0,0) node {$43$};
\draw[right] (x12)++(0,0) node {$34$};
\draw[right] (x10)++(-0.05,0) node {$24$};
\draw[below] (x11)++(0,0.) node {$31$};
\draw[left] (x7)++(0.,0) node {$42$};
\draw[above] (x6)++(0,0) node {$13$};
\end{scope}

\begin{scope}[xshift=0cm, yshift=0cm] 
\coordinate(x1) at (0,0);
\coordinate(x2) at (1,0);
\coordinate(x3) at (2,0);
\coordinate(x4) at (3,0);
\coordinate(x5) at (0,1);
\coordinate(x6) at (1,1);
\coordinate(x7) at (2,1);
\coordinate(x8) at (3,1);
\coordinate(x9) at (0,2);
\coordinate(x10) at (1,2);
\coordinate(x11) at (2,2);
\coordinate(x12) at (3,2);
\coordinate(x13) at (0,3);
\coordinate(x14) at (1,3);
\coordinate(x15) at (2,3);
\coordinate(x16) at (3,3);
\draw (x1) -- (x2) -- (x6) -- (x5) -- (x1);
\draw (x2) -- (x3) -- (x4) -- (x8) -- (x7) -- (x3);
\draw (x8) -- (x12) -- (x16) -- (x15) -- (x11) -- (x12);
\draw (x15) -- (x14) -- (x13) -- (x9) -- (x10) -- (x14);
\draw (x5) -- (x9);
\draw (3,0) -- (4,0);
\draw (0,3) -- (0,4);
\draw (7,3) -- (7,4);
\draw (3,7) -- (4,7);
\foreach \i in {1,...,16}
{ 
\draw(x\i)[fill=white] circle(\vr);
}
\draw(x14)[fill=black] circle(\vr);
\draw(x8)[fill=black] circle(\vr);
\draw(x1)[fill=black] circle(\vr);

\draw[below] (x1)++(0,-0.1) node {$111$};

\end{scope}

\begin{scope}[xshift=4cm, yshift=0cm] 
\coordinate(x1) at (0,0);
\coordinate(x2) at (1,0);
\coordinate(x3) at (2,0);
\coordinate(x4) at (3,0);
\coordinate(x5) at (0,1);
\coordinate(x6) at (1,1);
\coordinate(x7) at (2,1);
\coordinate(x8) at (3,1);
\coordinate(x9) at (0,2);
\coordinate(x10) at (1,2);
\coordinate(x11) at (2,2);
\coordinate(x12) at (3,2);
\coordinate(x13) at (0,3);
\coordinate(x14) at (1,3);
\coordinate(x15) at (2,3);
\coordinate(x16) at (3,3);
\draw (x1) -- (x2) -- (x6) -- (x5) -- (x1);
\draw (x2) -- (x3) -- (x4) -- (x8) -- (x7) -- (x3);
\draw (x8) -- (x12) -- (x16) -- (x15) -- (x11) -- (x12);
\draw (x15) -- (x14) -- (x13) -- (x9) -- (x10) -- (x14);
\draw (x5) -- (x9);
\foreach \i in {1,...,16}
{ 
\draw(x\i)[fill=white] circle(\vr);
}
\draw(x14)[fill=black] circle(\vr);
\draw(x8)[fill=black] circle(\vr);
\draw(x1)[fill=black] circle(\vr);
\draw[below] (x4)++(0,-0.1) node {$444$};

\end{scope}

\begin{scope}[xshift=0cm, yshift=4cm] 
\coordinate(x1) at (0,0);
\coordinate(x2) at (1,0);
\coordinate(x3) at (2,0);
\coordinate(x4) at (3,0);
\coordinate(x5) at (0,1);
\coordinate(x6) at (1,1);
\coordinate(x7) at (2,1);
\coordinate(x8) at (3,1);
\coordinate(x9) at (0,2);
\coordinate(x10) at (1,2);
\coordinate(x11) at (2,2);
\coordinate(x12) at (3,2);
\coordinate(x13) at (0,3);
\coordinate(x14) at (1,3);
\coordinate(x15) at (2,3);
\coordinate(x16) at (3,3);
\draw (x1) -- (x2) -- (x6) -- (x5) -- (x1);
\draw (x2) -- (x3) -- (x4) -- (x8) -- (x7) -- (x3);
\draw (x8) -- (x12) -- (x16) -- (x15) -- (x11) -- (x12);
\draw (x15) -- (x14) -- (x13) -- (x9) -- (x10) -- (x14);
\draw (x5) -- (x9);
\foreach \i in {1,...,16}
{ 
\draw(x\i)[fill=white] circle(\vr);
}
\draw(x14)[fill=black] circle(\vr);
\draw(x8)[fill=black] circle(\vr);
\draw(x1)[fill=black] circle(\vr);

\draw[above] (x13)++(0,0.1) node {$222$};
\end{scope}

\begin{scope}[xshift=4cm, yshift=4cm] 
\coordinate(x1) at (0,0);
\coordinate(x2) at (1,0);
\coordinate(x3) at (2,0);
\coordinate(x4) at (3,0);
\coordinate(x5) at (0,1);
\coordinate(x6) at (1,1);
\coordinate(x7) at (2,1);
\coordinate(x8) at (3,1);
\coordinate(x9) at (0,2);
\coordinate(x10) at (1,2);
\coordinate(x11) at (2,2);
\coordinate(x12) at (3,2);
\coordinate(x13) at (0,3);
\coordinate(x14) at (1,3);
\coordinate(x15) at (2,3);
\coordinate(x16) at (3,3);
\draw (x1) -- (x2) -- (x6) -- (x5) -- (x1);
\draw (x2) -- (x3) -- (x4) -- (x8) -- (x7) -- (x3);
\draw (x8) -- (x12) -- (x16) -- (x15) -- (x11) -- (x12);
\draw (x15) -- (x14) -- (x13) -- (x9) -- (x10) -- (x14);
\draw (x5) -- (x9);
\foreach \i in {1,...,16}
{ 
\draw(x\i)[fill=white] circle(\vr);
}
\draw(x14)[fill=black] circle(\vr);
\draw(x8)[fill=black] circle(\vr);
\draw(x1)[fill=black] circle(\vr);
\draw[above] (x16)++(0,0.1) node {$333$};

\end{scope}

\end{tikzpicture}
\caption{$S_{C_4}^1$ (left), $S_{C_4}^2$ (middle) and  $S_{C_4}^3$ (right); in the latter only labels of extreme vertices are shown. Shaded vertices present an $\iota$-set in the corresponding graph.}
\label{fig:S_C_4}
\end{center}
\end{figure}
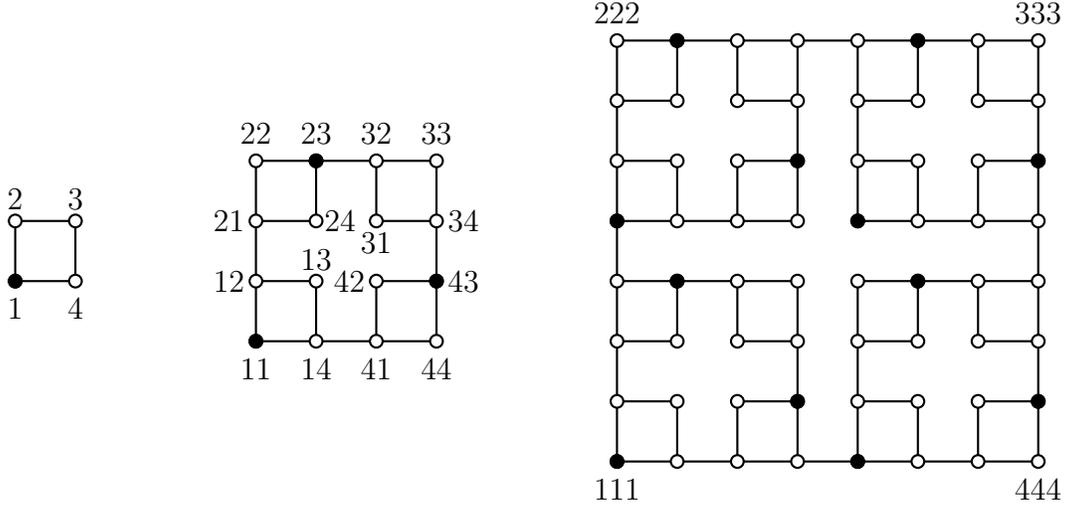

It is easy to check that the bounds from Theorems~\ref{thm:upperBound} and \ref{thm:lower} are sharp also for $G=C_4$. Note that in this case $\xi(S_{C_4}^2)=\io(S_{C_4}^2)=\io(S_{C_4}^2|Ex(S_{C_4}^2))=3$, see Fig.~\ref{fig:S_C_4}.

\begin{corollary}\label{prop:cycles}
For every $n\in\mathbb{N}-\{1\}$, $\iota(S_{C_4}^n)=3\cdot 4^{n-2}$.
   \end{corollary}

\section{Concluding remarks}
We conclude the paper with some directions for future research.

In Section~\ref{sec:prelim}, we considered the invariant $\gamma^i$, and proved that $\iota(G) \ge \gamma(G)-\gamma^{i}(G)$ holds for any graph $G$. The invariant $\gamma^i$
was in fact used in~\cite{aha-2009} for proving that chordal graphs satisfy Vizing's conjecture. For this purpose, the lower bound on the domination number of the Cartesian product of two graphs was proved in~\cite{aha-2009}, which involved the product of the domination number of one factor and the independence-domination number of the other factor. We wonder if the analogous inequality holds for the isolation number of the Cartesian product of the graphs.

\begin{problem}
\label{pr:gammai}
Is it  true that for any two graphs $G$ and $H$, 
$$\iota(G\,\Box\, H)\ge \max\{\gamma^i(G)\iota(H),\gamma^i(H)\iota(G)\}?$$
\end{problem}

Since $\gamma(G)\ge \gamma^i(G)$ holds in any graph $G$, the following inequality, if true, would be stronger than the one posed in Problem~\ref{pr:gammai}. Nevertheless, we were not able to prove or disprove either of the two. 

\begin{problem}
Is it  true that for any two graphs $G$ and $H$, 
$$\iota(G\,\Box\, H)\ge \max\{\gamma(G)\iota(H),\gamma(H)\iota(G)\}?$$
\end{problem}

If in the above problem we replace $\gamma$
with $\iota$ we obtain an inequality, which is also open, and can be considered as the isolation variant of the (conjectured) Vizing's inequality. Based on our preliminary investigations of this inequality, we think that it holds true, and pose it as follows. 

\begin{conjecture}
For any two graphs $G$ and $H$, 
$$\iota(G\,\Box\, H)\ge \iota(G)\iota(H).$$
\end{conjecture}

In Section~\ref{sec:Sierp}, we studied generalized Sierpi\'{n}ski graphs of arbitrary base graph $G$, and have bounded the isolation number of $S_G^t$, where $t\ge 2$, as follows:
 $$\io(S_G^2|Ex(S_G^2))\cdot n(G)^{t-2}\le \io(S_G^t) \leq \xi(S_G^2)\cdot n(G)^{t-2}.$$
The bounds have the factor $n(G)^{t-2}$ in common, while they differ in the functions $\io(S_G^2|Ex(S_G^2))$ and $\xi(S_G^2)$, which interestingly both rely on the generalized Sierpi\'{n}ski graph $S_G^t$ with dimension $t=2$. It is natural to ask how far these two values can be, which leads us to the following question.

\begin{problem}
Is there an integer $C\in \mathbb{N}$ such that $\xi(S_G^2)-\io(S_G^2|Ex(S_G^2))<C$ holds for all graphs $G$. 
\end{problem}

On the other hand, even more interesting is the question in which graphs $G$ the equality $\xi(S_G^2)=\io(S_G^2|Ex(S_G^2))$ holds. As proved in Section~\ref{sec:Sierp} such are all complete graphs $K_n$ and cycle $C_4$, which lead to the exact values of the isolation numbers of generalized Sierpi\'{n}ski graphs of these graphs.

\section*{Acknowledgements}
B.B., T.D. and A.T. acknowledge the financial support of the Slovenian Research and Innovation Agency (research core funding No.\ P1-0297, projects N1-0285, N1-0431, and the bilateral project between Slovenia and the USA entitled ``Domination concepts in graphs, project No. BI-US/24-26-036'').

\section*{Declaration of interests}
 
The authors declare that they have no conflict of interest. 

\section*{Data availability}
 
Our manuscript has no associated data.

\end{document}